\documentclass[12pt, reqno]{amsart}
\usepackage{amsmath, amsthm, amscd, amsfonts, latexsym, amssymb, graphicx, color}
\usepackage[bookmarksnumbered, colorlinks, plainpages]{hyperref}

\makeatletter \oddsidemargin.9375in \evensidemargin \oddsidemargin
\newtheorem{thm}{Theorem}[section]
\newtheorem{lem}[thm]{Lemma}
\newtheorem{prop}[thm]{Proposition}
\newtheorem{cor}[thm]{Corollary}
\theoremstyle{definition}
\newtheorem{defn}[thm]{Definition}
\newtheorem{exa}[thm]{Example}

\theoremstyle{remark}
\newtheorem{rem}[thm]{Remark}
\numberwithin{equation}{section}

\begin{document}

\setcounter{page}{1}


\title[M-central Armendariz rings]{Central Armendariz rings relative to a monoid}
\author[Z. Sharifi ]{Z. Sharifi}
\thanks{{\scriptsize
\hskip -0.4 true cm MSC(2010):16S36;
 }}
\begin{abstract} In this paper, the notion of central Armendariz rings relative to a monoid is introduced which is a
generalization of central Armendariz rings and investigate their
properties. It is shown that if $R$ is central reduced, then $R$
is $M$-central Armendariz for a u.p.-monoid $M$. For a monoid $M$
and ring $R$, we prove if $R$ is an $M$-central Armendariz, then
either $R$ is commutative or $M$ is cancellative. Various
examples which illustrate and delimit the results of this paper
are provided.

\end{abstract}

\maketitle
\section{Introduction}
All rings considered here are associative and unitary.  Rege and
Chhawchharia \cite{rege} introduced the notion of an Armendariz
ring. A ring $R$ is called {\it Armendariz} if whenever
polynomials $f(x) = a_0 + a_1x + ... + a_nx^n, g(x) = b_0 + b_1x +
... + b_mx^m \in  R[x]$ satisfy $f(x)g(x) = 0$, then $a_ib_j = 0$
for each $0\leq i\leq n$, $0\leq  j\leq m$. The name ``Armendariz
ring " was chosen because Armendariz \cite[Lemma 1]{arm} had shown
that a reduced ring (i.e, a ring without nonzero nilpotent
elements) satisfies this condition. Some properties of Armendariz
rings and theire generalizations  have been studied in
\cite{rege},  \cite{arm},  \cite{and}, \cite{ghks}, \cite{Zh}  and
\cite{lee1}. In \cite{L2}, Liu studied a generalization of
Armendariz rings, which is called $M$-Armendariz rings, where $M$
is a monoid. A ring R is called {\it $M$-Armendariz} (or
Armendariz relative to $M$) if whenever $\alpha = a_1g_1 + ...+
a_ng_n, \beta = b_1h_1 + ... + b_mh_m\in R[M]$, satisfy
$\alpha\beta = 0$, then $a_ib_j = 0$ for each $i, j$. Some
generalizations of Armendariz rings relative to a monoid can be
seen in \cite{Hashemi}, \cite{Hashe},  \cite{Hash}, \cite{Wang}
and \cite{Z}.

 According to Agayev, et.al \cite{N2}, a ring R is called
{\it central Armendariz} if whenever two polynomials $f(x) = a_0
+ a_1x + ... + a_nx^n, g(x) = b_0 + b_1x + ... + b_mx^m \in R[x]$
satisfy $f(x)g(x)=0$ then $a_ib_j \in C(R)$ for all $i,j$. In
this paper, we introduce the notion of $M$-central Armendariz
rings which are a common generalization of $M$-Armendariz rings
and central Armendariz rings. It is easy to see that the concept
of $M$-central Armendariz rings is related not only to the ring
$R$ but also to the monoid $M$. It is shown that if $R$ is
central reduced, then $R$ is $M$-central Armendariz for a
u.p.-monoid $M$. For a monoid $M$ and ring $R$, we prove if $R$ is
an $M$-central Armendariz, then either $R$ is commutative or $M$
is cancellative. It is clear that every $M$-Armendariz ring is
$M$-central. It is shown that the converse is not true in general
and the converse is hold if $R$ is a p.p.-ring and $M$  a strictly
totally ordered monoid. We end this paper with some applications
of $M$-central Armendariz rings to   show  there is a strong
connection between Baer and  p.p.-rings with their monoid  rings.

\section{M-central Armendariz rings}

In this section, central $M$-Armendariz rings  are introduced as
a generalization of $M$-Armendariz rings.

\begin{defn} Let $M$ be a monoid. A ring $R$ is called an
$M$-central Armendariz ring, if whenever elements
$\alpha=a_1g_1+a_2g_2+...+a_mg_m\in R[M]$ and
$\beta=b_1h_1+b_2h_2+...+b_nh_n\in R[M]$ satisfies
$\alpha\beta=0$, then $a_ib_j\in C(R)$ for each $1\leq i\leq m$
and $1\leq j\leq n$.
\end{defn}

In the following for a monoid $M$, $e$ always stands  for the
identity element of $M$.

\begin{rem} (1) If $M=\{e\}$, then every ring is $M$-central
Armendriz.

(2) Commutative rings are $M$-central Armendariz for each monoid
$M$.

(3) If $S$ is a semigroup with multiplication $st=0$ for each
$s,t\in S$ and $M=S^1$, then any noncommutative ring is not
$M$-central Armendariz.

(4) Let $M=(\mathbb{N}\cup\{0\},+)$. Then a ring $R$ is
$M$-central Armendariz if and only if $R$ is central Armendariz.

(5) Every $M$-Armendariz ring is $M$-central Armendariz. But the
converse is not true (see Example \ref{E2}).
\end{rem}

Recall that a ring $R$ is called central reduced if every
nilpotent element of $R$ is central \cite{Ung}. Let $P(R)$ denote
the prime radical and $Nil(R)$ the set of all nilpotent elements
of the ring $R$. The ring $R$ is called 2-primal if $P(R) =
Nil(R)$ (See namely \cite{Hi} and \cite{HJP}). Also a ring $R$ is
called nil-Armendariz relative to a monoid $M$ if whenever
elements $\alpha=a_1g_1+a_2g_2+...+a_mg_m\in R[M]$ and
$\beta=b_1h_1+b_2h_2+...+b_nh_n\in R[M]$ satisfies
$\alpha\beta\in Nil(R)[M]$, then $a_ib_j\in Nil(R)$ for each
$1\leq i\leq m$ and $1\leq j\leq n$ \cite{Hash}.

Recall that a monoid $M$ is called a u.p.-monoid (unique product
monoid) if for any two nonempty finite subsets $A, B \subseteq  M$
there exists an element $g \in M$ uniquely presented in the form
$ab$ where $a \in A$ and $b\in B$ (see \cite{Bir}).

\begin{thm} Let $M$ be a u.p.-monoid and $R$ a central reduced
ring. Then $R$ is a $M$-central Armendariz ring.
\end{thm}
\begin{proof} If $R$ is a central reduced ring, then $R$ is
2-primal by \cite[Theorem 2.15]{Ung}. Hence $Nil(R)$ is an ideal
of $R$. Hence by \cite[Proposition 2.1]{Hash}, $R$ is
nil-Armendariz relative to $M$. If
$\alpha=a_1g_1+a_2g_2+...+a_mg_m\in R[M]$ and
$\beta=b_1h_1+b_2h_2+...+b_nh_n\in R[M]$ satisfies
$\alpha\beta=0$, then $a_ib_j\in Nil(R)$ for each $i,j$. As $R$
is central reduced, $a_ib_j\in C(R)$ for each $i,j$.
\end{proof}

\begin{thm} Let $M$ be a monoid and $R$  a ring. If $R$ is an
$M$-central Armendariz, then either $R$ is commutative or $M$ is
cancellative.
\end{thm}
\begin{proof} Suppose $M$ is not cancellative. Hence $m,g,h\in M$
are such that $mg=mh$ and $g\neq h$. Then for any $r\in R$ we
have $(rm)(1g-1h)=0$. As $R$ is $M$-central Armendariz, $r\in
C(R)$. Hence $R$ is commutative.
\end{proof}

\begin{prop}\label{P1} For a ring $R$ and monoid $M$ with $|M|\geq 2$, the
following are equivalent:

(1) $R$ is $M$-central Armendariz;

(2) $R$ is Abelian, $fR$ and $(1-f)R$ are $M$-central Armendariz
for any idempotent $f\in R$;

(3) There is a central idempotent $f\in R$ such that $fR$ and
$(1-f)R$ are $M$-central Armendariz.
\end{prop}
\begin{proof} $(1)\Rightarrow (2)$ Clearly, subrings of any
$M$-central Armendariz ring are $M$-central Armendariz. It
suffices to show $R$ is Abelian. Let $f=f^2\in R$. Let $e$ be
identity of $M$ and $e\neq g\in M$. Then
$(fe-fr(1-f)g)((1-f)e-fr(1-f)g)=0$ implies that $fr(1-f)$ is
central, because $R$ is $M$-central Armendariz. Hence
$fr(1-f)=0$. Similarly, $(1-f)rf=0$. Therefore $f$ is central.

$(2)\Rightarrow (3)$ is clear.

$(3)\Rightarrow (1)$ Let $f$ be a central idempotent of $R$ and
$\alpha=a_1g_1+a_2g_2+...+a_mg_m,~\beta=b_1h_1+b_2h_2+...+b_nh_n\in
R[M]$ satisfies $\alpha\beta=0$. Clearly, we can prove $fa_ifb_j$
and $(1-f)a_i(1-f)b_j$ are central for each $1\leq i\leq m$ and
$1\leq j\leq n$. As $R=fR\oplus (1-f)R$ and $f$ is central,
$a_ib_j=fa_ib_j+(1-f)a_ib_j$ is central for each $1\leq i\leq m$
and $1\leq j\leq n$. Thus $R$ is $M$-central Armendariz.
\end{proof}

\begin{cor}\label{c1} \cite[Proposition 3.2]{L2} Every $M$-Armendariz ring with $|M|\geq 2$ is Abelian.
\end{cor}

The next example shows that the converse of corollary \ref{c1} is
not true in general.

\begin{exa} (1) Let $R$ be a noncommutative domain. Then $R$
is Abelian. Let  $S$ be  a semigroup with multiplication $st=0$
for each $s,t\in S$ and $M=S^1$. Then $R$ is not $M$-central
Armendariz.

(2) Let $M=(\mathbb{N}\cup\{0\},+)$ and
$$R=\left\{\left(\begin{array}{cc}
 a&b\\c&d
   \end{array}\right): a,b,c,d\in \mathbb{Z}~{\rm and}~a\equiv
   d,~b\equiv c\equiv 0 ~({\rm mod~ 2})~\right\}.$$
By \cite[Example 2.2]{N2}, $R$ is an Abelian ring which is not
$M$-central Armendariz.
\end{exa}

Let $(M,\leq)$ be an ordered monoid. If for any $g, g^\prime, h
\in  M$, $g < g^\prime$ implies that $gh < g^\prime h$ and $hg <
hg^\prime$ , then $(M,\leq)$ is called a strictly ordered monoid.

A ring $R$ is called right principal projective (it or simply,
right p.p.-ring) if the right annihilator of an element of $R$ is
generated by an idempotent. Clearly,  $M$-Armendariz rings are
$M$-central Armendariz. In the next theorem, we prove  that the
converse is true  if the ring is a right p.p.-ring.

\begin{thm}\label{T1} Let $R$ be a right p.p.-ring and $M$ be a strictly totally
ordered monoid. If $R$ is $M$-central Armendariz, then $R$ is
$M$-Armendariz.
\end{thm}
\begin{proof}  Let $\alpha=a_1g_1+a_2g_2+...+a_mg_m\in R[M]$ and
$\beta=b_1h_1+b_2h_2+...+b_nh_n\in R[M]$ be such that
$\alpha\beta=0$,  $g_1<g_2<...<g_m$  and $h_1<h_2<...<h_n$. We
will use transitive induction on strictly totally ordered set
$(M, \leq)$ to show that $a_ib_j=0$ for each $1\leq i\leq m$ and
$1\leq j\leq n$. Since $h_1\leq h_j$ for all $1\leq j\leq n$,
$g_1<g_j$ implies  $g_1h_1< g_ih_1\leq g_ih_j$ for each $1\leq
i\leq m$. Hence, if there exists $1\leq i\leq m$ and $1\leq j
\leq n$ such that $g_ih_j=g_1h_1$, then $g_1=g_i$ and $h_j=h_1$.
Therefore $a_1b_1=0$.

Now, suppose that $w\in M$ is such that for any $g_i$ and $h_j$
with $g_ih_j<w$, $a_ib_j=0$. We will show that $a_ib_j=0$ for any
$g_i$ and $h_j$ with $g_ih_j=w$. Set $X=\{(g_i, h_j):
g_ih_j=w\}$. Then $X$ is a finite set. We write $X$ as
$\{(g_{i_t},h_{j_t}):t=1,2,...,k\}$ such that
$g_{i_1}<g_{i_2}<...<g_{i_k}$. Since $M$ is cancellative,
$g_{i_1}=g_{i_2}$ and $g_{i_1}h_{j_1}=g_{i_2}h_{j_2}=w$ imply
$h_{j_1}=h_{j_2}$. Since $\leq $ is a strict order,
$g_{i_1}<g_{i_2}$ and $g_{i_1}h_{j_1}=g_{i_2}h_{j_2}=w$ imply
$h_{j_2}<h_{j_1}$. Thus we have $h_{j_k}<...<h_{j_2}<h_{j_1}$. Now
\begin{equation}\label{e1}
 \sum_{(g_i,h_j)\in X}a_ib_j=\sum_{t=1}^ka_{i_t}b_{j_t}=0.
 \end{equation}

For any $t\geq 2$, $g_{i_1}h_{j_t}<g_{i_t}h_{j_t}=w$, and so by
induction hypothesis, we have $a_{i_1}b_{j_t}=0$.   Since $R$ is
a right p.p.-ring, ${\rm r}_{R}(a_{i_t})=e_tR$ for some idempotent
element $e_t$ of $R$. Since $R$ is $M$-central Armendariz, $R$ is
Abelian by Proposition \ref{P1}. Hence $e_t\in C(R)$. Since
$b_{j_t}\in {\rm r}_R(a_{i_1})$ for each $t\geq 2$,
$b_{j_t}e_1=b_{j_t}$. By multiplying (\ref{e1}) by $e_1$ from the
right, we have
$a_{i_1}b_{j_1}e_1+a_{i_2}b_{j_2}e_1+...+a_{i_k}b_{j_k}e_1=
a_{i_1}e_1b_{j_1}+a_{i_2}b_{j_2}+...+a_{i_k}b_{j_k}=0$. Hence
\begin{equation}\label{e2}
a_{i_2}b_{j_2}+...+a_{i_k}b_{j_k}=0.
\end{equation}
 For any $t\geq 3$, $g_{i_2}h_{j_t}<g_{i_t}h_{j_t}=w$. So by
 induction hypothesis, we have $a_{i_2}b_{j_t}=0$. Hence
 $b_{j_t}e_2=b_{j_t}$. By multiplying (\ref{e2}) by $e_2$ from the
 right, we have $a_{i_3}b_{j_3}+...+a_{i_k}b_{j_k}=0$. Continuing
 this process yield $a_{i_k}b_{j_k}=0$. Thus (\ref{e2}) has the
 form $a_{i_1}b_{i_1}+...+a_{i_{k-1}}b_{j_{k-1}}=0$. As above, we
 have
 $a_{i_{k-1}}b_{j_{k-1}}=...=a_{i_2}b_{j_2}=a_{i_1}b_{j_2}=0$.

 Therefore by transitive induction, $a_ib_j=0$ for any $i,j$. Thus
 $R$ is $M$-central Armendariz.
\end{proof}

In the following example, it is shown that the condition "right
p.p.-ring" in Theorem \ref{T1} is not superfluous.

\begin{exa}\label{E2} Let $\mathbb{Z}_2$ be the field of integers modulo 2
and $$R=\{a_0+a_1i+a_2j+a_3k:a_i\in \mathbb{Z}_2 ~for
~i=0,1,2,3\}$$ be the Hamiltonian  quaternions over
$\mathbb{Z}_2$. Then $R$ is not a p.p.-ring by \cite[Example
1]{Huh}. Since $R$ is a commutative ring, it is $M$-central
Armendariz for each monoid $M$. Let $M$ be a monoid with $|M|\geq
2$, $e\neq g\in M$ and $\alpha=(1+i)e+(1+j)g$. Then $\alpha^2=0$,
but $(1+i)(1+j)\neq 0$ which implies that  $R$ is not
$M$-Armendariz.
\end{exa}

It was shown in \cite[Theorem 2.6]{N2} that if $I$ is a reduced
ideal of $R$ such that $R/I$ is a central  Armendariz ring, then
$R$ is central Armendariz. Here we have the following result,
which is a generalization of this Theorem.

\begin{thm}\label{T2} Let $M$ be a strictly totally ordered monoid and $I$
is an ideal of $R$. If $I$ is a reduced ring and $R/I$ is
$M$-central Armendariz, then $R$ is $M$-central Armendariz.
\end{thm}
\begin{proof} Let $a,b\in R$ and $ab=0$. Then by a similar
argument in the proof of \cite[Theorem 2.6]{N2}, we have
$bIa=aIb=0$.  Let $\alpha=a_1g_1+a_2g_2+...+a_mg_m\in R[M]$ and
$\beta=b_1h_1+b_2h_2+...+b_nh_n\in R[M]$ be such that
$\alpha\beta=0$,  $g_1<g_2<...<g_m$  and $h_1<h_2<...<h_n$. We
will use transitive induction on strictly totally ordered set
$(M, \leq)$ to show that $a_iIb_j=b_jIa_i=0$ for each $1\leq i\leq
m$ and $1\leq j\leq n$. By analogy with the proof of Theorem
\ref{T1}, we can show $a_1b_1=0$. Hence $a_1Ib_1=b_1Ia_1=0$.

Now, suppose that $w\in M$ is such that for any $g_i$ and $h_j$
with $g_ih_j<w$, $a_iIb_j=b_jIa_i=0$. We will show that
$a_iIb_j=b_jIa_j=0$ for any $g_i$ and $h_j$ with $g_ih_j=w$. Set
$X=\{(g_i, h_j): g_ih_j=w\}$. Then $X$ is a finite set. We write
$X$ as $\{(g_{i_t},h_{j_t}):t=1,2,...,k\}$ such that
$g_{i_1}<g_{i_2}<...<g_{i_k}$. Since $M$ is cancellative,
$g_{i_1}=g_{i_2}$ and $g_{i_1}h_{j_1}=g_{i_2}h_{j_2}=w$ imply
$h_{j_1}=h_{j_2}$. Since $\leq $ is a strict order,
$g_{i_1}<g_{i_2}$ and $g_{i_1}h_{j_1}=g_{i_2}h_{j_2}=w$ imply
$h_{j_2}<h_{j_1}$. Thus we have $h_{j_k}<...<h_{j_2}<h_{j_1}$. Now
\begin{equation}\label{e3}
 \sum_{(g_i,h_j)\in X}a_ib_j=\sum_{t=1}^ka_{i_t}b_{j_t}=0.
 \end{equation}

For any $t\geq 2$, $g_{i_1}h_{j_t}<g_{i_t}h_{j_t}=w$, and so by
induction hypothesis, we have
$a_{i_1}Ib_{j_t}=b_{j_t}Ia_{i_1}=0$. Since
$a_{i_1}Ia_{i_t}b_{j_t}\subseteq a_{i_1}Ib_{j_t}=0$,
$a_{i_1}Ia_{i_t}b_{j_t}=0$ for $t\geq 2$. By multiplying
(\ref{e3}) from the left by $a_{i_1}I$, we have
$a_{i_1}Ia_{i_1}b_{j_1}=0$, because $a_{i_1}Ia_{i_t}b_{j_t}=0$
for $t\geq 2$. Therefore $(b_{j_1}Ia_{i_1})^3=0$ because
$a_{i_1}b_{j_1}Ia_{i_1}b_{j_1}\subseteq
a_{i_1}Ia_{i_1}b_{j_1}=0$. Since $I$ is reduced,
$b_{j_1}Ia_{i_1}=0$. Also $a_{i_1}Ib_{j_1}=0$.

For any $t\geq 3$, $g_{i_2}h_{j_t}<g_{i_t}h_{j_t}=w$. So by
 induction hypothesis, we have
 $a_{i_2}Ib_{j_t}=b_{j_t}Ia_{i_2}=0$. Since $a_{i_1}Ib_{j_1}=0$,
 $(a_{i_2}Ia_{i_1}b_{j_1})^2=0$. Hence $a_{i_2}Ia_{i_1}b_{j_1}=0$.
 By multiplying (\ref{e3}) from left to $a_{i_2}I$, we have
 $a_{i_2}Ia_{i_2}b_{j_2}=0$. As above
 $a_{i_2}Ib_{j_2}=b_{j_2}Ia_{i_2}=0$. By continuing this process,
 we have $a_{i_t}Ib_{j_t}=b_{j_t}Ia_{i_t}=0$ for each $t=1,2,...,
 k$. Therefore  by transfinite induction, $a_iIb_j=b_jIa_j=0$.

Note that in $(R/I)[M]$, $(\overline{a_1}g_1 + · · · +
\overline{a_n}g_m)(\overline{b_1}h_1 + · · · + \overline{b_m}h_m)
= 0$. Since $R/I$ is $M$-central Armendariz,
$\overline{a_i}\overline{b_j}\in C(R/I)$. Thus $a_ib_jr - ra_ib_j
\in I$ for any $i,j$ and $r\in R$. Therefore $(a_ib_jr -
ra_ib_j)^3=0$. As $I$ is reduced, $a_ib_jr=ra_ib_j$ for each
$r\in R$ and $i,j$. Hence $R$ is $M$-central Armendariz.
\end{proof}

Recall that a monoid $M$ is called torsion-free if the following
property holds: if $g, h \in M$ and $k \geq 1$ are such that $g^k
= h^k$, then $g = h$.

\begin{cor} Let $M$ be a commutative, cancellative and torsion-free monoid.
If $R/I$ is $M$-central Armendariz for some ideal $I$ of $R$ and
I is a reduced ring, then $R$ is $M$-central Armendariz.
\end{cor}
\begin{proof} If M is commutative, cancellative and torsion-free, then by \cite{Rib} there exists
a compatible strict total ordered $\leq$ on $M$. Now the results
follows from Theorem \ref{T2}.
\end{proof}

\begin{rem}\label{R2} Let $R$ be any ring and $ n \geq 2$.  Consider the ring $M_n(R)$ of $n \times
n$ matrices and the ring $T_n(R)$ of $n \times n$ upper triangular
matrices over $R$. Then the  rings $M_n(R)$ and $T_n(R)$  are not
abelian. By Proposition \ref{P1}, these rings are not $M$-central
Armendariz for each monoid $M$.
\end{rem}

The next example shows that if $I$ is an ideal of $R$, $R/I$ and
$I$ are $M$-central Armendariz for a u.p.- monoid $M$, then $R$
is not $M$-central Armendariz in general.

\begin{exa} Let F be a field and consider $R=T_2(F)$, which is not $M$-central Armendariz, for a u.p.- monoid $M$ by Remark \ref{R2} with $|M|\geq 2$.
 It can be seen  that $R/I$ and $I$ are $M$-central Armendariz for some nonzero proper ideal $I$ of $R$.
 Assume that
 $I=\left(
\begin{array}{cc}
F&F\\0&0
   \end{array}\right)$. Then $R/I\cong F$ and so $R/I$ is $M$-Armendariz. Hence $R/I$ is $M$-central Armendariz.
 We prove that $I$ is $M$-central Armendariz.
Now let $\alpha=\sum_{i=1}^{n}\left(
\begin{array}{cc}
a_i&b_i\\0&0
   \end{array}\right)g_i$ and $\beta=\sum_{j=1}^{m}\left(
\begin{array}{cc}
c_j&d_j\\0&0
   \end{array}\right)h_j$ be nonzero elements of $I[M]$ such that $\alpha\beta=0$. From the isomorphism $T_2(F)[M]\cong T_2(F[M])$ defined by:

$$\sum_{i=1}^{n}\left(
\begin{array}{cc}
a_i&b_i\\0&c_i
   \end{array}\right)g_i\longrightarrow \left(
\begin{array}{cc}
\sum_{i=1}^{n}a_ig_i&\sum_{i=1}^{n}b_ig_i\\0&\sum_{i=1}^{n}c_ig_i
   \end{array}\right) $$
we have $\alpha_1\beta_1=\alpha_1\beta_2=0$, where
$\alpha_1=\sum_{i=1}^{m}a_ig_i$, $\beta_1=\sum_{j=1}^{n}c_jh_j$
and $\beta_2=\sum_{j=1}^{n}d_jh_j\in F[M]$.  As $F$ is reduced
and $M$ is a u.p.-monoid, $F$ is $M$-Armendariz by
\cite[Proposition 1.1]{L2}. Hence $a_ic_j=a_id_j=0$ for each
$i,j$. Therefore $$\left(
\begin{array}{cc} a_i&b_i\\0&0
   \end{array}\right) \left(
\begin{array}{cc}
c_j&d_j\\0&0
   \end{array}\right)=0$$ for each $i,j$. This implies that $I$ is
   $M$-Armendariz, and so it is $M$-central Armendariz.
\end{exa}

\begin{lem} Let $M$ be a cyclic group of order $n \geq 2$ and $R$ a noncommutative ring with $0
\neq 1$. Then $R$ is not $M$-central Armendariz.
\end{lem}
\begin{proof} Assume  that $M = \{e, g, g^2, · · · , g^{n-1}\}$ and $a\in R-C(R)$. Let  $\alpha = ae+ag+ag^2+· ·
·+ag^{n-1}$ and $\beta = 1e + (-1)g$. Then $\alpha\beta= 0$. Since
$a\not\in C(R)$, R is not M-central Armendariz.
\end{proof}
The proof of the next lemma is straightforward.
\begin{lem}  Let $M$ be a monoid and $N$ a submonoid of $M$. If $R$ is
$M$-central Armendariz, then $R$ is N-central Armendariz.
\end{lem}

\begin{prop} Let $M$ be a cancellative monoid and $N$ an ideal of $M$. If $R$ is
$N$-central Armendariz, then $R$ is $M$-central Armendariz.
\end{prop}
\begin{proof} Let $\alpha=a_1g_1+a_2g_2+...+a_mg_m\in R[M]$ and
$\beta=b_1h_1+b_2h_2+...+b_nh_n\in R[M]$ be such that
$\alpha\beta=0$. Let $g\in N$. Then $g_ig, h_jg\in N$ for each
$i,j$. Also $g_sg\neq g_tg$ and $h_sg\neq h_tg$ for each $s\neq
t$. Now from $(\alpha)g(\beta)g=(\sum_{i=1}^m
a_ig_ig)(\sum_{j=1}^nb_jh_jg)=0$, we have $a_ib_j\in C(R)$,
because $R$ is $N$-central Armendariz.
\end{proof}

Let $T(G)$ be the set of elements of finite order in an Abelian
group $G$. Then $T(G)$ is a fully invariant subgroup of $G$. $G$
is said to be torsion-free if $T(G) = \{e\}$.

\begin{prop} Suppose $G$ is  a finitely generated Abelian group. Then $G$ is torsion-free if and only if there
exists a ring $R$ with $|R|\geq 2$ such that $R$ is G-central
Armendariz.
\end{prop}
\begin{proof} See \cite[Theorem 1.14]{L2}.
\end{proof}

In \cite{Kap}, Baer rings are introduced as rings in which the
right (left) annihilator of every nonempty subset is generated by
an idempotent. We end this paper with some applications of
$M$-central Armendariz rings to   show  there is a strong
connection between Baer and  p.p.-rings with their monoid  rings.

\begin{thm}\label{T3} Let $M$ be a strictly totally ordered monoid with $|M|\geq
2$ and $R$  an $M$-central Armendariz ring. Then $R$ is right
p.p.-ring if and only if $R[M]$ is right p.p.-ring.
\end{thm}
\begin{proof} Let $R$ be a right p.p.-ring. By Theorem \ref{T1}, $R$
is $M$-Armendariz, because $R$ is $M$-central Armendariz. Hence
$R[M]$ is a right p.p.-ring by \cite[Theorem 3.4]{L2}.

Conversely, let $R[M]$ be a right p.p.-ring and $a\in R$. Then
there exists an idempotent $f=f_1g_1+f_2g_2+...+f_mg_m\in R[M]$
such that ${\rm r}_{R[M]}(a)=fR[M]$. We can suppose that
$g_1<g_2<...<g_m$. If there exist $1\leq i,j \leq m$ such that
$g_ig_j=g_1^2$, then $g_1\leq g_i$ and $g_1\leq g_j$. If $g_1< g_i
$, then $g_1^2<g_ig_1\leq g_ig_j=g_1^2$, a contradiction. Thus
$g_1=g_i$. Similarly, $g_1=g_j$. Hence from $(1-f)f=0$, we have
$(1-f_1)f_1=0$. Therefore $f_1^2=f_1$. As $afR[M]=0$,
$af_1g_1+...+af_mg_m=0$. Since  $g_i\neq g_j$ for each $1\leq
i\neq j\leq m$, $af_1=0$. Thus $f_1R\subseteq {\rm r}_R(a)$. Now,
let $r\in {\rm r}_R(a)$. Then $r\in fR[M]$. This implies that
$(1-f)r=0$ and so $(1-f_1)r=0$. Hence $r\in f_1R$. Therefore
$r_R(a)=f_1R$. Thus $R$ is a right p.p.-ring.
\end{proof}

\begin{thm} Let $M$ be a strictly totally ordered monoid with $|M|\geq
2$ and $R$  an $M$-central Armendariz ring. Then $R$ is a Baer
 ring if and only if $R[M]$ is a  Baer ring.
\end{thm}
\begin{proof} Let $R$ be a Baer ring. By Theorem \ref{T1}, $R$
is $M$-Armendariz, because $R$ is $M$-central Armendariz. Hence
$R[M]$ is a Baer ring by \cite[Theorem 3.5]{L2}.

Conversely, let $R[M]$ be a Baer ring  and $W$ be a subset of $R$.
Since $R[M]$ is Baer,  there exists $e^2 = e = e_1g_1 + e_2g_2 +
. . . + e_mg_m \in  R[M]$ such that $r_{R[M]}(W) = eR[M]$.
Similar the proof of Theorem \ref{T3}, $e_1^2=e_1$. As $we=0$ for
each $w\in W$, $we_1=0$ for each $w\in W$. Therefore
$e_1R\subseteq r_R(W)$. Now, let $r\in r_R(W)$. Then from
$r_R(W)\subseteq r_{R[M]}(W) = eR[M$ we have $(1-e)r=0$. Hence
$r\in e_1R$. Thus $r_R(W)=e_1R$ and so $R$ is Baer.
\end{proof}
\section*{Acknowledgments}
This research was supported by Islamic Azad University, Bushehr
Branch.

\bibliographystyle{amsplain}

\bigskip
\bigskip

{\footnotesize {\bf Zakaria Sharifi};

 {Department of Mathematics},

{Bushehr Branch, Islamic Azad University} {Bushehr, Iran.}

 {\tt Email:
sharifizakaria@gmail.com}

\end{document}